\documentclass[11pt, draft]{article}

\usepackage{amsmath, amsthm, amssymb, amsfonts,bm}
\usepackage{rotating}
\usepackage{enumerate, color}
\usepackage{algpseudocode, algorithm, algorithmicx}

\newcommand{\Input}{\textbf{Input: }}
\newcommand{\Output}{\textbf{Output: }}

\DeclareMathOperator{\linearspan}{span}
\DeclareMathOperator{\rank}{rank}
\DeclareMathOperator{\Tr}{Tr}

\DeclareMathOperator{\inner}{Inner}
\DeclareMathOperator{\inder}{Inder}
\DeclareMathOperator{\aut}{Aut}
\DeclareMathOperator{\End}{End}
\DeclareMathOperator{\ad}{ad}
\DeclareMathOperator{\obj}{Obj}

\newtheorem{theorem}{Theorem}
\newtheorem{lemma}[theorem]{Lemma}

\newtheorem{proposition}[theorem]{Proposition}
\newtheorem{definition}[theorem]{Definition}

\newtheorem{example}[theorem]{Example}

\title{Quadratic $2$-step Lie algebras:\\
Computational algorithms and classification}

\author{Pilar Benito,  Daniel de-la-Concepci\'on,\\
Jorge Rold\'an-L\'opez, Iciar Sesma}

\date{\quad}

\begin{document}
\maketitle\vspace{-1.5cm}


\begin{abstract}
Taking into account the theoretical results and guidelines given in this work, we introduce a computational method to construct any 2-step nilpotent quadratic algebra of $d$ generators. Along the work we show that the key of the classification of this class of metric algebras relies on certain families of skewsymmetric matrices. Computational examples for $d\leq 8$ will be given. 
\end{abstract}

{\parindent= 0em \small  {\bf Keywords:} Lie algebra, invariant bilinear form, nilpotent, free nilpotent, derivation, automorphism. }

{\parindent=10em \small  \sl }

\medskip

{\parindent=0em \small {\bf Classification MSC 2010:} 17B01, 17B30}


\section {Introduction}
 
A bilinear form $\varphi$ on a nonassociative algebra $A$ is said to be invariant if and only if
\begin{equation}\label{invariant}
\varphi(xy,z)=\varphi(x,yz)
\end{equation}
for all $x,y,z\in A$. The pair $(A,\varphi)$ is called quadratic (also metric) algebra in case $\varphi$ is invariant, symmetric ($\varphi(x,y)=\varphi(y,x)$) and nondegenerate ($\varphi(x,y)=0$ for all   $y\in A$ implies $x=0$). The main subject of this paper is the classification of quadratic two step ($2$-step in the sequel) nilpotent Lie algebras. Some computational algorithms attached to the classification results will be given.\medskip

A Lie algebra $(\mathfrak{g},[x,y])$ is said to be $2$-step nilpotent (or metabelian) of type $d$ in case $\mathfrak{g}^3=\linearspan\,\langle[[xy]z]:x,y,z\in \mathfrak{g}\rangle=0$ and $d=\dim\mathfrak{g}-\dim\mathfrak{g}^2$ where $\mathfrak{g}^2=\linearspan\,\langle[xy]:x,y,z\in \mathfrak{g}\rangle$. It is well known that the type of any $2$-step nilpotent Lie algebra is $\geq 2$.\medskip

Following \cite{NoRe97}, the classification of quadratic $2$-step Lie algebras is equivalent to the classification of  alternating trilinear forms, which is an open problem. Ovando (2007) shows the existence of real quadratic $2$-step Lie algebras of arbitrary type $d\geq 3$ with $d=4$ as the only exception. In \cite{O07}, it is proved that the algebras in  this class can be achieved from any injective homomorphism $\rho\colon\mathfrak{v}\to \mathfrak{so}(\mathfrak{v}, \langle\cdot,\,\cdot\rangle)$ of a vector space $\mathfrak{v}$ equipped with an inner product $\langle\cdot,\,\cdot\rangle$. In addition, the map $\rho$ must satisfy the relation
$\rho(v)(v)=0$ for all $v\in \mathfrak{v}$. The last condition can be rewritten as,
\begin{equation}\label{ro}
\rho(v)(u)+\rho(u)(v)=0.
\end{equation}

The main ideas in \cite{O07} and the results provided in \cite{BeCoLa17} allow us to establish some alternative results on the existence and isomorphisms of quadratic $2$-step nilpotent Lie algebras over arbitrary fields of characteristic $0$. From the proofs along this paper we develop some algorithms that we have implemented in the computer system Mathematica.\medskip 

We start with the model of $2$-step free nilpotent Lie algebra of type $d=\dim \mathfrak{v}$ (here $\mathfrak{v}$ denotes any vector space)$$\mathfrak{n}_{d,2}(\mathfrak{v})=\mathfrak{v}\oplus \Lambda^2\mathfrak{v},$$as it is given in \cite{Gau73}. Then we introduce a family of skewsymmetric matrices  $\{A_1,\dots, A_d\}$ of order $d\times d$ such that:
\begin{enumerate}[\quad a)]\label{skewmatrices}
\item The $i$-th column of $A_i$ is null,
\item For every $j>i$, the $j$-th column of $A_i$ is the additive inverse of the $i$-th column of $A_j$.
\end{enumerate}
Through these families of matrices, it is possible to describe the whole set of invariant forms $B$ of $\mathfrak{n}_{d,2}(\mathfrak{v})$ for which  $\mathfrak{n}_{d,2}(\mathfrak{v})^\perp\subseteq \mathfrak{n}_{d,2}(\mathfrak{v})^2$. According to \cite{BeCoLa17}, any quadratic $2$-step nilpotent Lie algebra appears as an hommomorphic image of $\mathfrak{n}_{d,2}$ of the form $\mathfrak{n}_{d,2}/\ker B$ where $\ker B=\mathfrak{n}_{d,2}(\mathfrak{v})^\perp$.\medskip

The paper is divided into 4 sections from number 2. Section 2 is devoted to the general theory of quadratic Lie algebras. This Section also includes a categorical approach of quadratic nilpotent Lie algebras given by means of bilinear invariant forms of free nilpotent Lie algebras.
This approach allow us to identify the classes of isometric isomorphism in Section 3, where quadratic $2$-step nilpotent algebras are explicitly described by using basis and structure constants. In the final Section 4, examples and computational algorithms are included.\medskip




From now on, any vector space  will be of finite dimension over an arbitrary field $\mathbb{K}$ of characteristic $0$. Basics notions and facts on Lie algebras follows from \cite{Ja62} and \cite{Hu72}.


\section{Generalities and examples}\label{section2}
A Lie algebra $\mathfrak{n}$ is said to be nilpotent if there is an integer $k$ such that $\mathfrak{n}^{k}=0$, where $\mathfrak{n}^k$ is defined inductively as $\mathfrak{n}^1 =\mathfrak{n}$, $\mathfrak{n}^k=[\mathfrak{n}^{k-1}, \mathfrak{n}]$. We call $t$ the \emph{index of nilpotency} of $\mathfrak{n}$ in case $\mathfrak{n}^{t+1}=0$ and $\mathfrak{n}^t\neq 0$. Then we say that $\mathfrak{n}$ is $t$-nilpotent or also $t$-step nilpotent. 
\medskip

The type of a nilpotent Lie algebra $\mathfrak{n}$ is defined as the codimension of $\mathfrak{n}^2$ in $\mathfrak{n}$. According to Section 1, Corollary 1.3,  in \cite{Gau73}, a set $\mathfrak{m}= \{ x_1, x_2, \dots, x_d \}$ generates $\mathfrak{n}$ if and only if $\{ x_1+\mathfrak{n}^2,\dots x_d+\mathfrak{n}^2\}$ is a basis of $\mathfrak{n}/\mathfrak{n}^2$. So, the type of a Lie algebra is the cardinality of any linear independent set $\mathfrak{m}$ such that $\mathfrak{n}=\mathfrak{t}\oplus \mathfrak{n}^2$, where $\mathfrak{t}$ is the linear span of $\mathfrak{m}$. The above condition implies that the set $\mathfrak{m}$ generates $\mathfrak{n}$ as an algebra. Therefore,  the elements $x_i\in \mathfrak{m}$ can be viewed as a \emph{minimal set of generators} of $\mathfrak{n}$.
\medskip

Given an arbitrary Lie algebra $\mathfrak{g}$ with product $[x,y]$, since $[x,y]=-[y,x]$ the condition \eqref{invariant} that defines an invariant form $\varphi$ on $\mathfrak{g}$ can be rewritten as
\begin{equation}\label{invarianteLie}
\varphi([x,y],z)+\varphi(y,[x,z])=0.
\end{equation}
This is equivalent to
\begin{equation}\label{so}
\inner\, \mathfrak{g}\subseteq \mathfrak{so}(\mathfrak{g},\varphi),
\end{equation}
where $\inner\, \mathfrak{g}=\linearspan\,\langle\ad\, x:x\in \mathfrak{g}\rangle$ is the set of inner derivations of $\mathfrak{g}$ and 
$$\mathfrak{so}(\mathfrak{g},\varphi)=\{f\in \End(\mathfrak{g}): \varphi(f(x),y)+\varphi(x,f(y))=0\}
$$
is the orthogonal Lie algebra attached to the bilinear form $\varphi$.\medskip

A quadratic abelian Lie algebra is just a vector space endowed with a nondegenerate bilinear form. Any semisimple Lie algebra with its Killing form, $B(x,y)=\Tr(\ad\ x\circ ad\ y)$, is quadratic. In fact, the non degeneration of this trace form $B(x,y)$ characterizes (characteristic zero) the class of semisimple Lie algebras according to Cartan's Criterion.  The structure of quadratic algebras is not so clear for non-semisimple and non-abelian Lie algebras.\medskip

Following \cite{TsWa57}, any quadratic Lie algebra $(\mathfrak{g},\varphi)$ satisfies (here $Z(\mathfrak{g})$ denotes the center of $\mathfrak{g}$):
\begin{equation}\label{condicionortogonal}
Z(\mathfrak{g})^\perp=\mathfrak{g}^2\quad and \quad Z(\mathfrak{g})=(\mathfrak{g}^2)^\perp.
\end{equation}
Hence, $\dim \mathfrak{g}=\dim \mathfrak{g}^2+\dim Z(\mathfrak{g})$. The dimension of $Z(\mathfrak{g})\cap\mathfrak{g}^2$ is called the \emph{isotropic index of $\mathfrak{g}$}. The algebra $\mathfrak{g}$ is said to be \emph{reduced} in case $Z(\mathfrak{g})\subseteq \mathfrak{g}^2$, and the pair $(r,s)$, where $r=\dim \mathfrak{g}^2$ and $s=\dim Z(\mathfrak{g})$, is called the (bi-)type of $\mathfrak{g}$.


\begin{lemma}[Tsou and Walker, 1956]\label{TsouWal} Any non reduced and non abelian quadratic Lie algebra $(\mathfrak{g},\varphi)$ decomposes as an orthogonal direct sum of proper ideals, $\mathfrak{g}=\mathfrak{g}_1\oplus \mathfrak{a}$, where $\varphi=\varphi_1\perp \varphi_2$ and $(\mathfrak{g}_1,\varphi_1)$ is a quadratic reduced Lie algebra and $(\mathfrak{a},\varphi_2)$ is a quadratic abelian algebra.
\end{lemma}

From (\ref{condicionortogonal}), it is immediate that Generalized Heisenberg Algebras (GHA) are not quadratic\footnote{A GHA algebra is defined as a Lie algebra $\mathfrak{h}$ such that $\mathfrak{h}^2=Z(\mathfrak{h})=\mathbb{K}\cdot z$.}. The real $5$-dimensional free $3$-step nilpotent Lie algebra (see Subsection 2.1 for a general definition)$$\mathfrak{n}_{2,3}(\mathbb{R}^2)=\mathbb{R}^2\oplus \Lambda^2 \mathbb{R}^2\oplus \frac{\mathbb{R}^2\otimes \Lambda^2 \mathbb{R}^2}{\Lambda^3 \mathbb{R}^2}$$admits two non isometric invariant metrics. In fact, the number of nonisometrically isomorphic quadratic structures on $\mathfrak{n}_{2,3}(\mathbb{K}^2)$ is equal to the cardinality of the quotient $\mathbb{K}^*/(\mathbb{K}^*)^2$, where $\mathbb{K}^*$ is the multiplicative group of the field $\mathbb{K}$. So there is only one quadratic structure on the complex algebra $\mathfrak{n}_{2,3}(\mathbb{C})$ and infinite in case $\mathfrak{n}_{2,3}(\mathbb{Q}^2)$, where $\mathbb{Q}$ is the rational field. In contrast, up to isometries, there is only one indecomposable\footnote{A quadratic algebra $(A,\varphi)$ is decomposable in case $A$ splits into an orthogonal direct sum of (quadratic) ideals. Otherwise, $(A,\varphi)$ is called indecomposable.} and quadratic $2$-step Lie algebra of type 3 over any field of characteristic zero. This algebra is the $6$-dimensional free nilpotent $\mathfrak{n}_{3,2}(\mathbb{K}^3)=\mathbb{K}^3\oplus \Lambda^2 \mathbb{K}^3$ as it is proved in \cite{BeCoLa17}\footnote{In fact, $\mathfrak{n}_{3,2}(\mathbb{K}^3)$ and $\mathfrak{n}_{2,3}\mathbb{K}^2$ are the only free nilpotent Lie algebras that admit a quadratic structure according to \cite{BaOv12}.}.\medskip

The next definition and example are remarkable for the classification of $2$-step nilpotent Lie algebras. The example is included in \cite{O07}. 

\begin{definition}
Let $\mathfrak{v}$ denotes a vector space and $\lambda=\pm 1$. We define the hyperbolic $\lambda$-bilinear form on $\mathfrak{v}$ to be the form $\varphi_{\mathfrak{v}}^\lambda$ on $\mathbb{H}(\mathfrak{v})=\mathfrak{v}\oplus \mathfrak{v}^*$, defined by
\begin{equation}\label{hiperbolica}
\varphi_\mathfrak{v}^\lambda (v_1+f_1,v_2+f_2)=f_1(v_2)+\lambda f_2(v_1),
\end{equation}
for all $v_1,v_2\in \mathfrak{v}$ and $f_1,f_2\in \mathfrak{v}^*$. If $\lambda=1$, the form $\varphi_{\mathfrak{v}}^1$ is symmetric, and if $\lambda=-1$, it is an alternating (skew symmetric) bilinear form. A bilinear form $\psi$ is called a hyperbolic bilinear form if it is isometric to  $\varphi_{\mathfrak{v}}^\lambda$ for some vector space $\mathfrak{v}$ and some $\lambda=\pm 1$.
\end{definition}

\begin{example}\label{alcotangente} Let $(\mathfrak{v},\,\langle\cdot,\,\cdot\rangle,\,\rho)$ be the a triple where $(\mathfrak{v},\,\langle\cdot,\,\cdot\rangle)$ is a metric vector space over $\mathbb{R}$ (so $\langle\cdot,\,\cdot\rangle$ is an inner product on $\mathfrak{v}$) and $\rho\colon \mathfrak{v}\to\mathfrak{so}(\mathfrak{v})$ denotes an injective linear map that satisfies condition~\eqref{ro}. For every $u,v\in \mathfrak{v}$, let $f_{u,v}: \mathfrak{v} \to \mathbb{R}$ define as$$f_{u,v}(w)=\langle\rho(w)(u), v\rangle=-\langle\rho(u)(w), v\rangle)$$for all $w\in \mathfrak{v}$. Consider now the vector space $\mathfrak{n}(\mathfrak{v},\rho)=\mathfrak{v}\oplus\mathfrak{v}^*$ endowed with the canonical hyperbolic metric
$$
\varphi_\mathfrak{v}^1 (v_1+f_1,v_2+f_2)=f_1(v_2)+f_2(v_1).
$$In $\mathfrak{n}(\mathfrak{v},\rho)$ we define the bracket product
$$
[u+g,v+h]=[u,v]=f_{u,v},
$$for $u,v\in \mathfrak{v}$ and $g,h\in \mathfrak{v}^*$. Since $\rho(a)\in \mathfrak{so}(\mathfrak{v})$, the product $[\cdot,\,\cdot]$ is skewsymmetric. It is also clear that $[[u,v],w]=0$. Hence, $\mathfrak{n}(\mathfrak{v},\rho)$ is a $2$-step Lie algebra. From condition \eqref{ro} of $\rho$ we get that $\varphi_\mathfrak{v}^1$ is an invariant bilinear form. Therefore $(\mathfrak{n}(\mathfrak{v},\rho), \varphi_\mathfrak{v}^1)$ is a quadratic $2$-step Lie algebra. In fact $\mathfrak{n}(\mathfrak{v},\rho)^2=\mathfrak{v}^*=Z(\mathfrak{n}(\mathfrak{v},\rho))$ because of $\rho$ is injective. So the Lie algebra  $\mathfrak{n}(\mathfrak{v},\rho)$ is reduced.
\end{example}

Theorem 3.2 in \cite{O07} establishes that for every integer $m\geq 1$, if we provide $\mathbb{R}^m$ (consider as an abelian Lie algebra) with a metric $\Phi_m$, the real quadratic $2$-step nilpotent Lie algebras (up to isometries) are of the form $(\mathbb{R}^m\oplus \mathfrak{n}(\mathfrak{v}, \rho),\Phi_m\perp \varphi_\mathfrak{v}^1)$.


\subsection{Free nilpotent Lie algebras vs quadratic nilpotent Lie algebras}

Let $\mathfrak F \mathfrak L (d)$ be the free Lie algebra on a set of $d$ generators over the field $\mathbb{K}$. The free $t$-nilpotent Lie algebra on $d$ generators is denoted $\mathfrak{n}_{d,t}$ and defined as the quotient algebra
\begin{equation}\label{feedefinicion}
\mathfrak{n}_{d,t}= \mathfrak F \mathfrak L (d) / \mathfrak F \mathfrak L (d)^{t+1}.
\end{equation}
Any $t$-nilpotent Lie algebra $\mathfrak{n}$ of type $d$ over $\mathbb{K}$ is an homomorphic image of $\mathfrak{n}_{d,t}$. According to Proposition 1.4 and Proposition 1.5 in \cite{Gau73}, $\mathfrak{n}$ is isomorphic to the factor algebra $\mathfrak{n}_{d,t}/I$, where $I$ is an ideal of $\mathfrak{n}_{d,t}$ such that $I\subseteq\mathfrak{n}_{d,t}^2$ and $\mathfrak{n}_{d,t}^t\not\subseteq I$. (See \cite{GraGro90} for models of free nilpotent Lie algebras.)
\medskip

In \cite{BeCoLa17} it is introduced a new technique for constructing quadratic nilpotent Lie algebras out of invariant symmetric bilinear forms on free nilpotent Lie algebras. In the sequel we summarize briefly some of the results we will use along this paper. The next lemma follows from Proposition 4.1 in \cite{BeCoLa17}.

\begin{lemma}\label{reduccion}Let $\mathfrak{n}$ be the factor Lie algebra $\displaystyle{\dfrac{\mathfrak{n}_{d,t}}{I}}$ where $I$ is an ideal of $\mathfrak{n}_{d,t}$ such that $\mathfrak{n}_{d,t}^t \not\subseteq I \subseteq \mathfrak{n}_{d,t}^2$. Then, there exists a symmetric, invariant and non-degenerate bilinear form $\varphi$ on  $\mathfrak{n}$ if and only if  there exists a symmetric, invariant and non-degenerate bilinear form $\psi$ on $\mathfrak{n}_{d,t}$ such that $I=\mathfrak{n}_{d,t}^\perp$. The relation between $\varphi$ and $\psi$ is given by $\psi(a,b)=\varphi(a+I, b+I)$ for all $a,b\in \mathfrak{n}_{d,t}$.
\end{lemma}

Previous lemma reduces the classification of quadratic nilpotent Lie algebras to that of invariant bilinear forms on free nilpotent Lie algebras. Moreaver we can establish the following categorical approach:
\begin{itemize}
\item $\bf {NilpQuad_{d,t}}$ stands for the category whose objects are the quadratic $t$-nilpotent Lie algebras $(n,\,\varphi)$ of type $d$. The morphisms $f\colon(n,\,\varphi)\to (n',\,\varphi')$ are the \emph{isometric Lie homomorphisms}, i.e.: $$f([x,y])=[f(x),f(y)] \quad\text{and}\quad \varphi(x,y)=\varphi'(f(x), f(y)).$$ 

\item $\bf Sym_0(d,t)$ is the category whose objects are the symmetric invariant bilinear forms $\psi$ on $\mathfrak{n}_{d,t}$ for which $\ker  \psi \subseteq \mathfrak{n}_{d,t}^2$ and $\mathfrak{n}_{d,t}^t\nsubseteq \ker  \psi $. The  morphisms are isometric Lie homomorphisms of $\mathfrak{n}_{d,t}$ moduli the relation of equivalence
\begin{equation}\label{equivalencia}
f_1\sim f_2 \Longleftrightarrow (f_1-f_2)(\mathfrak{n}_{d,t})\subseteq \ker(\psi_2),
\end{equation}
where $f_i\colon(\mathfrak{n}_{d,t},\psi_1)\to(\mathfrak{n}_{d,t},\psi_2)$ for $i=1,2$.
\item ${\bf Q_{d,t}} \colon{ \bf Sym_0(d,t)} \to {\bf NilpQuad_{d,t}}$ is the functor that associates to each object $\psi$ in the category $\bf Sym_0(d,t)$, the object in $\bf {NilpQuad_{d,t}}$:
\begin{equation}
{\bf Q_{d,t}}(\psi) = (\mathfrak{n}_{d,t}/\ker  \psi, \varphi),\quad \varphi(a+\ker  \psi, b+\ker  \psi)= \psi(a,b).
\end{equation}

\end{itemize} 

In \cite{BeCoLa17} it is proved that the functor $\bf Q_{d,t}$ provides an equivalence between the categories $\bf Sym_0(d,t)$ and $\bf NilpQuad_{d,t}$. Moreover, there is a natural action of $\aut\, \mathfrak{n}_{d,t}$, the group of authomorphism of $\mathfrak{n}_{d,t}$, on the set of objects of the category $\bf Sym _0(d,t)$,
\begin{equation}\label{accion}
\aut\, \mathfrak{n}_{d,t}\times \obj ({\bf Sym _0(d,t)})\to \obj ({\bf Sym _0(d,t)}), \quad (\theta, \psi)\mapsto \psi_\theta
\end{equation}
where $\psi_\theta (x,y)=\psi(\theta(x),\theta(y))$. Using the functor $\bf Q_{d,t}$ and the action introduced in (\ref{accion}), from Corollary 4.3 and Lemma 4.4 in \cite{BeCoLa17}) we get the next lemma.

\begin{lemma}\label{equivalentes}
For all $\psi_1, \psi_2 \in \obj ({\bf Sym _0(d,t)})$, the following assertions are equivalent:
\begin{enumerate}[\quad (i)]
\item $\psi_1$ and  $ \psi_2$ are isomorphic objects in the category ${\bf Sym _0(d,t)}$.
\item ${\bf Q_{d,t}}(\psi_1)$ and ${\bf Q_{d,t}}(\psi_2)$ are isometrically isomorphic Lie algebras.
\item There exists an isometric automorphism $\theta\colon(\mathfrak{n}_{d,t},\psi_1)\to (\mathfrak{n}_{d,t},\psi_2)$.
\end{enumerate}
\end{lemma}

\begin{lemma}\label{isomorfismo} For all $\psi\in \obj ({\bf Sym _0(d,t)})$, the set $Orb_{\aut\, \mathfrak{n}_{d,t}}(\psi)=\{\psi_\theta: \theta\in \aut\, \mathfrak{n}_{d,t}\}$ equals to the set of bilinear invariant symmetric forms that are isomorphic to $\psi$ in the category ${\bf Sym _0(d,t)}$. Therefore the number of orbits of the action given in \emph{(\ref{accion})} is exactly the number of isometric isomorphism classes of quadratic $t$-nilpotent Lie algebras of type $d$.
\end{lemma}

In the following section we recombine the ideas and results in \cite{O07} and \cite{BeCoLa17} to get an explicit classification of quadratic $2$-step nilpotent Lie algebras.


\section{ Two step quadratic Lie algebras}


According to \cite{ElKaMe08}, a nondegenerate bilinear form $\varphi:\mathfrak{v}\times \mathfrak{v}\to \mathbb{K}$ is called metabolic if there is a totally isotropic subspace $\mathfrak{w}$ of $\mathfrak{v}$ of dimension equal to $\frac{1}{2}\dim\mathfrak{v}$. In characteristic different from $2$, metabolic and hyperbolic forms coincide as it is proved in Corollary 1.25 in \cite{ElKaMe08}.

For a given quadratic and reduced $2$-step nilpotent Lie algebra $(\mathfrak{n}, \varphi)$ we have that $Z(\mathfrak{n})=\mathfrak{n}^2=(\mathfrak{n}^2)^\perp$ because of (\ref{condicionortogonal}). Therefore  $\mathfrak{n}^2$ is a totally isotropic subspace of $\mathfrak{n}$ of dimension 
$\frac{1}{2}\dim\mathfrak{n}$. It follows that any nondegenerate invariant bilinear form attached to a quadratic and reduced $2$-step nilpotent Lie algebra is metabolic. Over the real field $\mathbb{R}$ metabolic and hyperbolic metrics coincides. This is the key in \cite{O07} to establish that the class of quadratic and reduced real $2$-step nilpotent Lie algebras are the algebras  given in Example \ref{alcotangente}.


Let $(\mathfrak{n}, \varphi)$ be a quadratic and reduced $2$-step nilpotent Lie algebra of type $d$. Since $\varphi$ is metabolic, for a given basis $\{z_1,z_2,\dots, z_d\}$ of $\mathfrak{n}^2$, there exists a set $\{v_1,v_2,\dots, v_d\}$ of orthogonal vectors such that $\varphi(z_i,v_j)=\delta_{ij}$. Therefore, $\mathcal{B}=\{v_1,v_2,\dots, v_d, z_1,z_2,\dots, z_d\}$ is an ordered basis of $\mathfrak{n}$ and  $\varphi$, in the basis $\mathcal{B}$, is determined by the the matrix:
\begin{equation}\label{hyperbolicmatrix}
  B_0=\left(
\begin{array}{cc}
\textbf{0}_{d\times d} & \mathbf{I}_{d\times d} \\
\mathbf{I}_{d\times d}& \mathbf{0}_{d\times d} \\
\end{array}
\right).
\end{equation}

On the other hand, $\varphi$ is related to some invariant form of the $2$-step free nilpotent algebra $\mathfrak{n}_{d,2}$. From \cite{Gau73}, we will use as model of $\mathfrak{n}_{d,2}$
\begin{equation}\label{model2step}
\mathfrak{n}_{d,2}(\mathbb{K}^d)=\mathbb{K}^d\oplus\Lambda^2\mathbb{K}^d, \quad [u_1+v_1\wedge w_1,u_2+v_2\wedge w_2]=u_1\wedge u_2.
\end{equation}
If we take a basis $\{u_1,\dots,u_d\}$ of $\mathbb{K}^d$, the set $\{u_i, [u_i,u_j]=u_i\wedge u_j:1\leq i\leq d,i<j\}$ is a basis of $\mathfrak{n}_{d,2}(\mathbb{K}^d)$ called  \emph{Hall basis} (see \cite{Ha50} for a complete definition).

\begin{definition}\label{dcuadratica} For any $d\geq 2$, a family $\{A_1,\dots,A_d\}$ of matrices of order $d\times d$ is called $d$-quadratic if the following properties are satisfied:
\begin{enumerate}[\quad (1)]
\item Every matrix $A_i$ is skewsymmetric ($A_i^t=-A_i$, where $A_i^t$ is the transpose matrix of $A_i$).
\item The $i$-th column of every $A_i$ is null.
\item For any $j>i$, the $j$-th column of $A_i$ is the additive inverse of the $i$-th column of $A_j$.
\item If $B_{i<j}$ denotes the submatrix of $A_i$ given by the set of all $j$-th columms of $A_i$ such that $i<j$, the  matrix
\begin{equation}\label{Qq}
    B(A_1,\dots,A_{d})=[B_{1<j}B_{2<j}\dots B_{d-1<j}],
\end{equation}
of order $d\times \frac{d(d-1)}{2}$, has maximal rank $d$. 
\end{enumerate}
\end{definition}

\begin{theorem}\label{clasifcuadraticas} For any $d\geq 2$,  the following assertions are equivalent:
\begin{enumerate}[\quad (i)]
\item $(\mathfrak{n},\varphi)$ is a quadratic and reduced $2$-step nilpotent Lie algebra of type $d$.
\item There exists a basis  $\mathcal{B}=\{v_1,\dots,v_d,z_1,\dots, z_d\}$ of $\mathfrak{n}$ for which the structure constants are determined by a family of $d$-quadratic matrices $\{A_{i}=(a^i_{kj}): 1\leq i\leq d\}$ in the following way: $[z_i,\mathfrak{n}]=0$ and $[v_i,v_j]=\sum_{k=1}^da^i_{kj}z_k$. Moreover, the bilinear form $\varphi$ is hyperbolic and, in the ordered basis $\mathcal{B}$, $\varphi$ is given by the the matrix $B_0$ described in \emph{(\ref{hyperbolicmatrix})}.


\end{enumerate}

\end{theorem}

\begin{proof}  Assume firstly that $(\mathfrak{n},\,\varphi)$ is quadratic and reduced. From (\ref{condicionortogonal}) we get $Z(\mathfrak{n})= \mathfrak{n}^2=(\mathfrak{n}^2)^\perp$. Hence we can take
a minimal set of generators $\{v_1,\dots, v_d\}$ of $\mathfrak{n}$ and a basis $\{z_1,\dots, z_d\}$ of  $\mathfrak{n}^2$ such that $\varphi(v_i,z_j)=\delta_{ij}$. Therefore, $\mathfrak{n}$ decomposes as the direct sum,
$$\mathfrak{n}=\mathfrak{v}\oplus \mathfrak{n}^2,$$where $\mathfrak{v}$ denotes the linear $\linearspan$ of $\{v_1,\dots, v_d\}$
and the matrix of $\varphi$ attached to the basis $\mathcal{B}=\{v_1,\dots, v_d,z_1,\dots, z_d\}$ is as in (\ref{hyperbolicmatrix}). Now, the inner derivation algebra of $\mathfrak{n}$, $\inner\,  \mathfrak{n}$, is generated by the set $\{\textrm{ad}\, v_i:1\leq i\leq n\}$. This algebra is abelian and $\dim\inner\,  \mathfrak{n}=d$ because of $\mathfrak{v}\cap Z(\mathfrak{n})=0$ (the map $\mathfrak{v}\to \inner\, \mathfrak{n}$, $v\mapsto \textrm{ad}\, v$ is one-one).\medskip

Following (\ref{so}), the form $\varphi$ is invariant if and only if $\inner\,\mathfrak{n}$ is a subalgebra of the orthogonal Lie algebra $\mathfrak{so}(\mathfrak{n},\varphi)$. We can decompose this algebra as the direct sum of subspaces:
$$
	\mathfrak{so}(\mathfrak{n},\varphi)=\mathfrak{so}(\mathfrak{v}\oplus \mathfrak{n}^2,\varphi)=\sigma_{\mathfrak{v}, \mathfrak{v}}\oplus\sigma_{\mathfrak{v}, \mathfrak{n}^2}\oplus\sigma_{\mathfrak{n}^2, \mathfrak{n}^2},
$$
where $\sigma_{\mathfrak{p},\mathfrak{q}}=\linearspan\,\langle\varphi(x,\,\cdot\,)y-\varphi(y,\,\cdot\,)x: x\in \mathfrak{p}, y\in \mathfrak{q}\rangle$. The algebra $\mathfrak{so}(\mathfrak{n},\varphi)$ is $\mathbb{Z}_2$-graduated by declaring  $\mathfrak{so}_0=\sigma_{\mathfrak{v}, \mathfrak{n}^2}$ as the even part and $\mathfrak{so}_1=\sigma_{\mathfrak{v}, \mathfrak{v}}\oplus \sigma_{\mathfrak{n}^2, \mathfrak{n}^2}$ as the odd part. Therefore, any $h\in \inner\,  \mathfrak{n}$ decomposes as $h=h_0+h_1$ where $h_i\in \mathfrak{so}_i$. Moreover, we have that
$$
h(\mathfrak{v})\subseteq \mathfrak{n}^2, \text{ and } \ h(\mathfrak{n}^2)=0.
$$Hence $d_0=0$ and the matrix representation of $d$ with respect to the basis $\mathcal{B}$ is
\begin{equation}\label{matrixder}
h=h_1=\left(
\begin{array}{cc}
\textbf{0}_{d\times d} & \mathbf{0}_{d\times d} \\
A& \mathbf{0}_{d\times d} \\
\end{array}
\right)=\mathbf{A}(h).
\end{equation}
Now,
\begin{equation}\label{equivalenciaortogonal}
h\in \mathfrak{so}(\mathfrak{n},\varphi)\Leftrightarrow \varphi(h(x),y)+\varphi(x,h(y))=0 \Leftrightarrow \mathbf{A}(h)^tB+B\mathbf{A}(h)=0.
    \end{equation}
Note that (\ref{equivalenciaortogonal}) is equivalent to $A^t=-A$. Let $A_i$ denote the skewsymmetric matrice of order $d\times d$ attached to $\mathbf{A}(\textrm{ad}\, v_i)$ for $i=1,\dots, d$. We claim that $\{A_1,\dots, A_{d}\}$ is a family of $d$-quadratic matrices. Conditions (1), (2) and (3) in Definition \ref{dcuadratica} follow from (\ref{matrixder}), (\ref{equivalenciaortogonal}) and the anticommutativity of the Lie bracket of $\mathfrak{n}$. Condition (4) follows from the basic fact that $\linearspan\,\langle h(v_i):i=1,\dots ,d, h\in \inner\, \mathfrak{n} \rangle=\mathfrak{n}^2$. This implies that the rank of $B(A_1,\dots,A_d)$ is exactly $d$.\medskip

For the converse, we note that the Lie bracket given in (ii) is anticommutative by using properties (2) and (3) of $d$-quadratic matrices. The triple product $[[x,y],z]$ is null in $\mathfrak{n}$, so, the Jacobi identity is trivial. Therefore, $(\mathfrak{n},[\cdot,\cdot])$ is a Lie algebra. From condition (4) of a $d$-quadratic family, we have that $0\neq\mathfrak{n}^2=\linearspan\,\langle z_1,\dots,z_d\rangle$.  Hence, $\mathfrak{n}$ is a $2$-step nilpotent Lie algebra of type $d$. The structure constants $a^i_{kj}$ of $\mathfrak{n}$ show us that the matrix representation of $\mathbf{A}(\ad\,v_i)$ is determine by the matrix $A_i$. According to condition (1) in Definition \ref{dcuadratica}, every matrix $A_i $ is skewsymmetric. Hence, the nondegenerate bilinear form $\varphi(v_i,z_j)=\delta_{ij}$ and $\varphi(v_i,v_j)=\varphi(z_i,z_j)=0$ is invariant because of (\ref{equivalenciaortogonal}). In order to prove that $\mathfrak{n}$ is reduced, we will test that $Z(\mathfrak{n})\cap\mathfrak{v}=0$, where $\mathfrak{v}=\linearspan\langle v_1,\dots,v_d \rangle $. Let $u=x_1v_1+\dots+x_dv_d$ be ($x_i$ scalars) and assume $u\in Z(\mathfrak{n})$. Denote $\mathbf{u}=(x_1,\dots,x_n)$ and note that $A_i\mathbf{u}=0$ for $i=1,\dots, d$. Since $A_i^t=-A_i$, we have that $[A_1\,\dots,A_d]^t\mathbf{u}=0$. From condition (4) of a $d$-quadratic family, the rank of $[A_1\,\dots,A_d]$ is $d$. Therefore $u=0$ which proves the equality $Z(\mathfrak{n})\cap\mathfrak{v}=0$.
\end{proof}

In the sequel, we denote as $(\mathfrak{n}(A_1,\dots,A_d), \varphi_{0})$ the quadratic Lie algebra attached to the family of $d$-quadratic matrices $\{A_1,\dots,A_d\}$ as it is described in item (ii) of Theorem \ref{clasifcuadraticas}. For any $n\geq 1$, we also denote $(\mathbb{K}^n,\Phi_n)$ the quadratic abelian Lie algebra of type $n$ where $\Phi_n$ is the nondegenerate bilinear form given by the inner product $\Phi_n(u,v)=\sum_{i=1}^n u_iv_i$.

\begin{example} The $2\times 2$ skewsymmetric matrices are of the form
$$
\left(
\begin{array}{cc}
0&a\\
-a&0\\
\end{array}
\right),\quad a\in \mathbb{K}.
$$So, there are no $2$-quadratic families of matrices. Hence, in characteristic zero, there are no $2$-step quadratic Lie algebras of type $2$.
\end{example}

\begin{example} Let $d=3$ be. The $3$-quadratic families of matrices are of the form $\{\mathbf{A}_1(a),$ $\mathbf{A}_2(a),$ $\mathbf{A}_3(a)\}$ where
$$
\mathbf{A}_1(a)=
\begin{pmatrix}
0&0&0\\
0&0&a\\
0&-a&0\\
\end{pmatrix},
\enspace\mathbf{A}_2(a)=
\begin{pmatrix}
0&0&-a\\
0&0&0\\
a&0&0\\
\end{pmatrix},
\enspace\mathbf{A}_3(a)=
\begin{pmatrix}
0&a&0\\
-a&0&0\\
0&0&0\\
\end{pmatrix}.
$$and  $0\neq a\in \mathbb{K}$. It is easily checked that any $2$-step nilpotent quadratic Lie algebra of type $3$ is reduced. Therefore the class of quadratic $2$-step Lie algebras of type $3$ are of the form $(n(\mathbf{A}_1(a),$ $\mathbf{A}_2(a),$ $\mathbf{A}_3(a)), \varphi_0)$. Following \cite{BeCoLa17}, any of them is isometrically isomorphic to the quadratic Lie algebra $(\mathfrak{n}_{3,2}(\mathbb{K}^3), \Psi_0)$, where $\mathfrak{n}_{3,2}(\mathbb{K}^3)=\mathbb{K}^3\oplus \Lambda^2\mathbb{K}^3$ and the matrix of $\Psi_0$ with respect to the basis $\{e_i, e_i\wedge e_j\}$ is

\noindent\emph{(\emph{here} $e_1=(1,0,\dots,0), e_2=(0,1,\dots,0), \dots, e_d=(0,0,\dots,1)$)}
$$
\left(
\begin{array}{cccccc}
0&0&0&0&0&-1\\
0&0&0&0&1&0\\
0&0&0&-1&0&0\\
0&0&-1&0&0&0\\
0&1&0&0&0&0\\
-1&0&0&0&0&0\\
\end{array}
\right).
$$
\end{example}

\begin{example}
For $d=4$, the existence of quadratic and reduced $4$-step Lie algebras is equivalent to the existence of a $4\times 6$ matrix of maximal rank $4$ with the following shape:
$$
B(A_1,A_2,A_3,A_4)=\left(
\begin{array}{cccccc}
0&0&0&-a&-b&-c\\
0&a&b&0&0&-d\\
-a&0&c&0&d&0\\
-b&-c&0&-d&0&0\\
\end{array}
\right).
$$
Any minor of order $4$ of this type of matrices is null. So, the rank of the matrix $B(A_1,A_2,A_3,A_4)$ is less or equal than $3$ and therefore in characteristic zero there are no $2$-step reduced quadratic Lie algebras of type $4$. Up to isometric isomorphisms the non reduced are of the form $$(\mathfrak{n}_{3,2}(\mathbb{K}^3)\oplus \mathbb{K}, \Psi_0 \perp \Phi_1).$$
\end{example}


From now on, our goal is to investigate the existence problem of quadratic Lie algebras $(\mathfrak{n}(A_1,\dots,A_d), \varphi_{0})$ for any arbitrary $d\geq 5$. We will also study the problem of isometric isomorphisms. The first problem is a simple exercise of linear algebra. For the second one, we will use the functorial relation introduced in Section \ref{section2}.\medskip

The next result ensures the existence of quadratic and reduced $2$-step Lie algebras of arbitrary type $d$ different from $1,2$ and~$4$ and and over any field $\mathbb{K}$ of characteristic zero. This result has also been established in \cite[]{O07}, Proposition 3.3.  We provide here an alternative proof.

\begin{proposition}\label{existence}
For any $d\neq 1,2,4$, there exist $d$-quadratics families of matrices.
\end{proposition}

\begin{proof}
From previous examples it is clear that $d\geq 3$ and $d\neq 4$.
Assume firstly $d>1$ is odd, and choose a $d\times d$ skewsymmetric matrix  $A_1$ of rank $d-1$, with its first row and its first column being null. Then, $N_1=\{u\in \mathbb{K}^d: uA_1=A_1u^t=0\}=\mathbb{K}\cdot e_1$ where $e_1=(1,0,\cdots,0)$. As $d>1$, we can take a second skewsymmetric matrix $A_2$ such that its first row is the additive inverse of the second row of $A_1$. Clearly, $e_1\notin \{u\in \mathbb{K}^d: uA_2=0\}$. Therefore, $\{u\in \mathbb{K}^d: u[A_1\,A_2]=0\}=0$. This implies that $[A_1\,A_2]$ is a matrix of maximal rank $d$. We can now extend the set $\{A_1, A_2\}$ to a $d$-quadratic family $\{A_1, A_2, \dots, A_d\}$ in an easy (and not unique) way.\medskip

In case $d>4$ be even, take a $d\times d$ skewsymmetric matrix $A_1$ with rows $r_1=r_d=(0,\dots,0)$ and rows $\,r_2,\dots,\,r_{d-1}$ being a set of linearly independent row vectors of $\mathbb{K}^d$. Therefore, the rank of $A_1$ is $d-2$. It is clear that $N_1=\{u\in \mathbb{K}^d: uA_1=0\}=\mathbb{K}\cdot e_1\oplus \mathbb{K}\cdot e_d$. Now, choose a second matrix $A_2$ with first row $-r_2=-(0,0,-s_3,\dots,-s_d)\neq 0$ and its second row being null. Since $d>4$, we can take the $3$-th and $4$-th rows of $A_2$ as $p_3=(s_3,0,0,0,\dots,1)$ and $p_4=(s_4,0,0,0,0,\dots,a)$ where $a\neq s_4/s_3$ if $s_3\neq 0$ or $a=0$ otherwise. Finally we add $j$-th rows for $j=5,\dots,d$ just to get $A_2$ as a skewsymmetric matrix. For the pair $A_1,A_2$, we claim that $N_2=\{u\in \mathbb{K}^d: u[A_1\,A_2]=0\}=0$. \medskip

Suppose there is $0\neq u\in N_2$. Since $u\in N_1$ we can write $u=t\cdot e_1+s\cdot e_d$ for some $(0,0)\neq(t,s) \in \mathbb{K}\times \mathbb{K}$. We note that $u\neq e_1$ because of the first row of $A_2$ is not null. Reescaling if necessary, we can assume $u=t_0\cdot e_1+e_d$. In particular, we have that
$$
u\cdot p_3^t=t_0s_3+1=0=u\cdot p_4^t=t_0s_4+a.
$$
Hence, $s_3\neq 0$ and $t_0=-\dfrac{1}{s_3}$ which implies $a=\dfrac{s_4}{s_3}$, a contradiction. This proves our claim. Now, $N_2=0$ implies that $[A_1\,A_2]$ is a matrix of maximal rank $d$ and, reasoning as in the odd case, we get our result.
\end{proof}

An element $\psi\in {\bf Sym_0(d,2)}$ is called reduced if and only if ${\bf Q_{d,2}}(\psi)=\dfrac{\mathfrak{n}_{d,2}}{\ker \psi}$ is a reduced Lie algebra. This assertion is equivalent to
\begin{equation}\label{conditon-dt-reduced}
\mathfrak{n}_{d,2}^2=\{x\in \mathfrak{n}_{d,2}:[x,\mathfrak{n}_{d,2}]\subseteq \ker  \psi\}.
\end{equation}
For any $d$-quadratic family $\{A_1,\dots,A_{d}\}$, we denote
\begin{equation}\label{matriz-dt-formabilineal}
\mathcal{Q}(A_1,\dots,A_{d})=\left(
\begin{array}{cc}
\textbf{0}_{d\times d} &B(A_1,\dots,A_{d})\\
B(A_1,\dots,A_{d})^t& \mathbf{0}_{\frac{d(d-1)}{2}\times \frac{d(d-1)}{2}}\\
\end{array}
\right),
\end{equation}
where $B(A_1,\dots,A_{d})$ is defined as in (\ref{Qq}).\medskip

Following Proposition 3 in \cite{Sa71}, for a given $\{v_1,\dots,v_d\}$ basis of a $d$-dimensional vector space $\mathfrak{v}$, any linear map $f\colon\mathfrak{v}\to \mathfrak{n}_{d,2}(\mathfrak{v})=\mathfrak{v}\oplus\Lambda^2\mathfrak{v}$, for which the vectors $f(v_1),\dots,f(v_d)$ are linearly independent, extends to the automorphism $\tau_f$ by declaring
\begin{equation}\label{ruleauto}
\tau_f([v_i,v_j])=\tau_f(v_i\wedge v_j)=f(v_i)\wedge f(v_j)=[f(v_i),f(v_j)].
\end{equation}
Even more, any automorphism of  $\mathfrak{n}_{d,2}$ is of this form. Hence, in the Hall basis $\mathcal{H}_{\mathfrak{v}}=\{v_i,v_i\wedge v_j:i=1,\dots,d, i<j\}$, the automorphisms of $\mathfrak{n}_{d,2}(\mathfrak{v})$ are represented by matrices of the form
\begin{equation}\label{automorfismos}
\tau_\mathfrak{s}(Q,X)=\left(\begin{array}{cc}
Q &\mathbf{0}_{d\times \frac{d(d-1)}{2}}\\
X& \hat{Q}\\
\end{array}
\right),
\end{equation}
where $X$  is a any matrix of order $\frac{d(d-1)}{2}\times d$, $Q$ is a regular matrix of order $d\times d$, and $\hat{Q}$ is a matrix completely determined from $Q$ by the rule \eqref{ruleauto}. In case $Q=(b_{ij})$, from a straightforward computation we get that
\begin{equation}\label{ruleQgorro}
\tau_f(v_i\wedge v_j)=\sum_{1\leq r<s\leq n}\,\det \left(\begin{array}{cc}
b_{ri} &b_{rj}\\
b_{si} &b_{sj}\\
\end{array}
\right)v_r\wedge v_s.
\end{equation}
The equation (\ref{ruleQgorro}) provides the entries of the matrix $\hat{Q}$ in terms of the entries of the matrix $Q$. 

\begin{theorem}
Let $\{A_1,\dots, A_d\}$ and $\{E_1,\dots, E_d\}$ be two families of $d$-quadratic matrices and let $(\mathfrak{n}(A_1,\dots, A_d),\varphi_0)$ and $(\mathfrak{n}(E_1,\dots, E_d),\psi_0)$ be the quadratic Lie algebras attached to them as it is described in Theorem \ref{clasifcuadraticas}. Then, the Lie algebras $(\mathfrak{n}(A_1,\dots, A_d),\varphi_0)$ and $(\mathfrak{n}(E_1,\dots, E_d),\psi_0)$ are isometrically isomorphic if and only if
there exists a regular $d\times d$ matrix $Q$ such that
$$
B(E_1,\dots, E_d)=Q^tB(A_1,\dots, A_d)\hat{Q},
$$
where $\hat{Q}$ is given in terms of $Q=(b_{ij})$ through the formula \eqref{ruleQgorro} and $B(E_1,\dots, E_d)$, $B(A_1,\dots, A_d)$ are as described in (\ref{Qq}).
\end{theorem}

\begin{proof}
Along the proof we will denote $\mathfrak{n}(A_1,\dots, A_d)$ and $\mathfrak{n}(E_1,\dots, E_d)$ as $\mathfrak{n}_A$ and $\mathfrak{n}_B$ respectively. According to Lemma \ref{reduccion}, $(\mathfrak{n}_A,\varphi_0)$ is isometrically isomorphic to $(\mathfrak{n}_{d,2}/\ker \varphi, \varphi)$ for some reduced invariant bilinear form $\varphi\in {\bf Sym _0(d,2)}$. Then, from Theorem \ref{clasifcuadraticas}, there exists a minimal generator set $\mathfrak{u}=\{u_1,\dots,u_d\}$ of $\mathfrak{n}_{d,2}$ of isotropic vectors and a totally isotropic central ideal $\mathfrak{c}=\linearspan\,\{z_1,\dots, z_d\}$ such that
$$
\mathfrak{n}_{d,2}=\mathfrak{v}\oplus \wedge^2\mathfrak{v}=\linearspan\,\langle u_1,\dots,u_d\rangle\oplus (\mathfrak{c}\oplus \ker  \varphi),
$$
$$
[u_i,u_j]=u_i\wedge u_j\equiv \sum_{k=1}^da^i_{kj}z_k\quad (\textrm{mod}\ \ker  \varphi),
$$
$$
\varphi(u_i,z_j)=\delta_{ij},
$$
and the set $\{a_{kj}^i\}$ of structure constants is determined by the $d$-quadratic family $\{A_1,\dots,A_d\}$. So, if we fixed a basis $\mathcal{B}$ of $\ker  \varphi=\mathfrak{n}_{d,2}^\perp$, the matrix of $\varphi$ attached to $\mathcal{B}'=\{u_1,\dots, u_d,z_1,\dots, z_d\}\cup \mathcal{B}$ is
$$
\mathcal{M}=\left(
\begin{array}{ccc}
\textbf{0}_{d\times d} & \mathbf{I}_{d\times d} &\mathbf{0}_{d\times \frac{d(d-3)}{2}}\\
\mathbf{I}_{d\times d}& \mathbf{0}_{d\times d} &\mathbf{0}_{d\times \frac{d(d-3)}{2}}\\
 \mathbf{0}_{ \frac{d(d-3)}{2}\times d} & \mathbf{0}_{ \frac{d(d-3)}{2}\times d}& \mathbf{0}_{ \frac{d(d-3)}{2}\times\frac{d(d-3)}{2}}\\
\end{array}
\right).
$$

For every $i=1,\dots, d$, the inner derivation $\textrm{ad}\, u_i$is represented by a matrix (respect to the basis $\mathcal{B}'$) of the form 
$$\left(
\begin{array}{ccc}
\textbf{0}_{d\times d} & \mathbf{0}_{d\times d} &\mathbf{0}\\
\mathbf{A_i}& \mathbf{0}_{d\times d} &\mathbf{0}\\
 \mathbf{C_i} &\mathbf{0}&\mathbf{0}\\
\end{array}
\right).$$
Let $D_{i<j}$ be the submatrix of $C_i$ given by the set of all $j$-th columms of $C_i$ such that $i<j$. Denote as $D$ the matrix $D=[D_{1<j}D_{2<j}\dots D_{d-1<j}]$ and let $P$ be the matrix:
$$
P=\left(
\begin{array}{cc}
\textbf{I}_{d\times d} & \mathbf{0}_{d\times \frac{d(d-1)}{2}} \\
\mathbf{0}_{d\times d}& B(A_1,\dots,A_{d}) \\
\mathbf{0}_{\frac{d(d-3)}{2}\times d}& D \\
\end{array}
\right)
$$
Clearly $P$ is a regular matrix because of $\mathfrak{n}_{d,2}^2=\linearspan\,\langle h(u_i):h\in \inder\, \mathfrak{n}_{d,2} \rangle$. In fact, $P$ provides the change of basis from the Hall basis $\mathcal{H}_{\mathfrak{u}}=\{u_i,u_i\wedge u_j\}$ to $\mathcal{B}'$. Therefore, the matrix of $\varphi$ attached to $\mathcal{H}_{\mathfrak{u}}$ is
$P^t\mathcal{M}P=\mathcal{Q}(A_1,\dots,A_d)
$ just as defined in (\ref{matriz-dt-formabilineal}). In an analogous way, the Lie algebra $(\mathfrak{n}_E,\psi_0)$ is isometrically isomorphic to $(\mathfrak{n}_{d,2}/\ker \psi, \psi)$ for some reduced bilinear form $\psi\in {\bf Sym _0(d,2)}$. From this isomorphism we get a Hall basis $\mathcal{H}_{\mathfrak{v}}=\{v_i,v_i\wedge v_j\}$ and a regular matrix $R$ such that $R^t\mathcal{M}R=\mathcal{Q}(E_1,\dots,E_d)$ is the matrix of $\psi$ attached to $\mathcal{H}_{\mathfrak{v}}$. The change of basis from $\mathcal{H}_\mathfrak{u}$ to $\mathcal{H}_{\mathfrak{v}}$ provides the authomorphism
$\tau(Q,X)$ of $\mathfrak{n}_{d,2}$ and the matrix that represents $\varphi$ in the basis $\mathcal{H}_{\mathfrak{v}}$ is
$$
\tau(Q,X)^t\mathcal{Q}(A_1,\dots,A_d)\tau(Q,X)=\left(\begin{array}{cc}
X^tB_A^tQ+Q^tB_AX &Q^tB_A\hat{Q}\\
\hat{Q}^tB_AQ& \mathbf{0}_{\frac{d(d-1)}{2}\times \frac{d(d-1)}{2}}\\
\end{array}
\right),
$$
where $B_A=B(A_1,\dots, A_d)$. Now, $(\mathfrak{n}_E, \psi_0)$ and $(\mathfrak{n}_A, \varphi_0)$ are isometrically isomorphic if and only if $\varphi$ and $\psi$ are isomorphic in the category ${\bf Sym _0(d,2)}$. From Lemma~\ref{isomorfismo}, the latter assertion is equivalent to the existence of an isometric automorphism $\theta\colon(\mathfrak{n}_{d,2},\psi)\to (\mathfrak{n}_{d,2},\varphi)$. Hence, $\psi=\varphi_\theta$. In the Hall basis $\mathcal{H}_{\mathfrak{v}}$, the automorphism $\theta$ is of the form $\tau(S,Y)$. Then,
$$
\tau(S,Y)^t\tau(Q,X)^t\mathcal{Q}(A_1,\dots,A_d)\tau(Q,X)\tau(S,Y)=\mathcal{Q}(E_1,\dots,E_d).
$$and
$$
\tau(QS,XS+\hat{Q}Y)^t\mathcal{Q}(A_1,\dots,A_d)\tau(QS,XS+\hat{Q}Y)=\mathcal{Q}(E_1,\dots,E_d).
$$
This implies $(QS)^tB_A\widehat{QS}=B_E$ which proves the result.
\end{proof}

Finally, the class of quadratic $2$-step Lie algebras are described in the following theorem.

\begin{theorem}\label{final} For every $d\geq 3$, there exist quadratic $2$-step Lie algebras of type $d$. Up to isometric isomorphisms the algebras in this class are of the form
$(\mathfrak{n}(A_1,\dots, A_{d_1})\oplus \mathfrak{a},\varphi\perp\phi)$ where $d=d_1+d_2$, $4\neq d_1\geq 3$, $\{A_1,\dots, A_{d_1}\}$ is a $d_1$-quadratic family of matrices and $(\mathfrak{a}, \phi)$, is a quadratic and abelian Lie algebra of dimension $d_2\geq 0$.
\end{theorem}

\begin{proof}
The result follows from Lemma \ref{TsouWal},  Theorem \ref{clasifcuadraticas} and Proposition \ref{existence}. 
\end{proof}


\section {Computational algorithms}

In this section we will introduce a computational method for constructing arbitrary (reduced) $2$-step quadratic nilpotent Lie algebras of type $d\geq 3$. The method relies in the theoretical results developed in previous Section 3, Theorems \ref{clasifcuadraticas} and \ref{final}. According to Definition \ref{dcuadratica}, we will designed several algorithms in order to build $d$-quadratric families of matrices ({\bf Algorithms 1 and 2}). Along the section we will use the following notation:
\begin{itemize}
\item $A[a,b]$ denotes the entry in row $a$ and column $b$ for a given matrix $A$;
\item $A[a::b,c::d]$ denotes the submatrix obtained from $A$ by using the whole set of rows from $a$ up to $b$ and the set of columns from $c$ up to $d$, both included.
\end{itemize}
Moreover we consider the first row/column is the one numbered as 1.\medskip


The matrices of a $d$-quadratic family $\{A_1,\dots, A_d\}$ are skew symmetric. From Theorem \ref{clasifcuadraticas} and its proof, every matrix $A_i$ determines an inner derivation $\ad v_i$ (adjoint derivation) of the quadratic Lie algebra $\mathfrak{n}(A_1,\dots, A_d)$. The set $\{v_1,\dots,v_d\}$ is a minimal generator set of $$\mathfrak{n}(A_1,\dots, A_d)=\linearspan\langle v_1,\dots, v_d, z_1,\dots, z_d\rangle,$$  and the adjoint $\ad v_i$ is represented by the matrix  ${\bf A}(\ad v_i)$ as it is described in (\ref{matrixder}). 
The computational procedure to obtain quadratic reduced Lie algebras splits into two steps:
\begin{itemize}
\item First we need to build the skewsymmetric family of matrices $A_1,\dots, A_d$ satisfying conditions (1), (2) and (3) of Definition \ref{dcuadratica} and the matrix $B(A_1,\dots, A_d)=[A_{1<j}A_{2<j}\dots A_{d-1<j}]$ associated to this family,
\item[] and
\item secondly, we need to verify that the rank of $B(A_1,\dots, A_d)$ is maximal. 
\end{itemize}


\subsection{Algorithm for generating a $\bm{d}$-quadratric family}

\begin{algorithm}[H]
	\caption{Skewsymmetric($d,\,s$) algorithm}
	
    \Input A natural $d$ indicating the dimension of the square matrix to generate.

	\Input An integer $s$ indicating the initial subindex of the variables to use.

    \Output A skewsymmetric $d\times d$ matrix whose variables are $a_{s}$, $a_{s+1}, \ldots$

  \begin{algorithmic}[1]
    \State Let $A$ be a empty $d\times d$ matrix
    \State Let $num = s$ \Comment{$num$ keeps the current variable value}
    \For {$i =1$ {\bf to} $i=d$}
    		\State Let $v = (a_{num}, \ldots, a_{num+d+i-1})$ a vector
    		\State $A[i,i+1 :: d] = v$
    		\State $A[i+1 :: d, i] = -v$.
     	\State $num = num + d-i$
	\EndFor
	
	\State \Return $A$
  \end{algorithmic}
\label{A:skewsymmetic}
\end{algorithm}

So, using this algorithm, for values $d=4$ and $s = 1$ the return matrix is
\begin{equation*}
\mathrm{Skewsymmetric}(4,1)=
\left(
\begin{array}{cccc}
0 & a_1 & a_2 & a_3 \\
-a_1 & 0 & a_4 & a_5 \\
-a_2 & -a_4 & 0 & a_6 \\
-a_3 & -a_5 & -a_6 & 0
\end{array}
\right).
\end{equation*}
To construct the adjoint matrices we will need to know hoy many subindexes we have used before, in order not to use them again. In a generic $d\times d$ matrix
\begin{equation*}
\sum_{i=1}^d i = \frac{d(d+1)}{2}
\end{equation*}
different variables are used. So, in the $m$-th adjoint matrix, due to the relationships among them, we will have used
\begin{equation*}
\sum_{i=d}^m \frac{i(i+1)}{2} = \frac{(1 + m - nd) (2 m + m^2 + d + m d + d^2)}{6}.
\end{equation*}
variables. We call this quantities varIn($d$) and varUntil($d,m$) respectively. With this notation
\begin{equation*}
\mathrm{varUntil}(d,m) = \sum_{i=d}^m \mathrm{varIn}(i).
\end{equation*}

\begin{algorithm}[H]
	\caption{Adjoint($d,\,i$) algorithm}
	
    \Input A natural $d$ indicating the dimension of the square adjoint matrix to generate.

	\Input An integer $i$ indicating which adjoint is it, assuming $i=1$ is the adjoint associated to the first element in the basis, and $i=d$ the one related to the last one.

    \Output The adjoint matrix associated to the element $x_i$ of the chosen basis.

  \begin{algorithmic}[1]
    \State Let $A_i$ be a empty $d\times d$ matrix
	\For {$j=1$ {\bf to} $j=i$} \Comment{The part deduced from previous adjoints}
		\State $A_i[1::d, j] = - \mathrm{Skewsymmetric}(d,j)[1::d, i]$
		\State $A_i[j, j::d] = - A_i[j::d,j]$
	\EndFor
    \State $A_i[i+1:d,i+1:d] = \mathrm{Skewsymmetric}(d-i, 1 + \mathrm{varUntil}(d-i, d-2))$
	
	\State \Return $A_i$
  \end{algorithmic}
\label{A:adjoint}
\end{algorithm}

Note that this algorithm is defined recursively. Although if it is needed an iterative version can be easily develop. In order to improve the efficiency it is highly recommended to store the already calculated adjoint matrices in order not to repeat operations.\medskip

Once we are able to obtain the adjoint matrices we can construct our $C[d] := B(A_1,\dots, A_d)$ matrix. This matrix is the union of certain columns of the adjoint matrices: from each adjoint, considering we are at the $i$-th adjoint, we take columns $i+1, i+2, \dots, d$.\medskip

Now we see a complete example of the calculus of matrix $B(A_1,\dots, A_d)$, for $A_i = \mathrm{Adjoint}(d,i)$ and $d = 5$.

\begin{equation*}
\begin{array}{cc}
 A_1 = \left(
\begin{array}{ccccc}
 0 & 0 & 0 & 0 & 0 \\
 0 & 0 & a_{1} & a_{2} & a_{3} \\
 0 & -a_{1} & 0 & a_{4} & a_{5} \\
 0 & -a_{2} & -a_{4} & 0 & a_{6} \\
 0 & -a_{3} & -a_{5} & -a_{6} & 0 \\
\end{array}
\right)&
A_2 = \left(
\begin{array}{ccccc}
 0 & 0 & -a_{1} & -a_{2} & -a_{3} \\
 0 & 0 & 0 & 0 & 0 \\
 a_{1} & 0 & 0 & a_{7} & a_{8} \\
 a_{2} & 0 & -a_{7} & 0 & a_{9} \\
 a_{3} & 0 & -a_{8} & -a_{9} & 0 \\
\end{array}
\right)
\end{array}
\end{equation*}
\begin{equation*}
A_3 = \left(
\begin{array}{ccccc}
 0 & a_{1} & 0 & -a_{4} & -a_{5} \\
 -a_{1} & 0 & 0 & -a_{7} & -a_{8} \\
 0 & 0 & 0 & 0 & 0 \\
 a_{4} & a_{7} & 0 & 0 & a_{10} \\
 a_{5} & a_{8} & 0 & -a_{10} & 0 \\
\end{array}
\right)
\end{equation*}
\begin{equation*}
\begin{array}{cc}
A_4 = 
\begin{pmatrix}
 0 & a_{2} & a_{4} & 0 & -a_{6} \\
 -a_{2} & 0 & a_{7} & 0 & -a_{9} \\
 -a_{4} & -a_{7} & 0 & 0 & -a_{10} \\
 0 & 0 & 0 & 0 & 0 \\
 a_{6} & a_{9} & a_{10} & 0 & 0 \\
\end{pmatrix}
&
A_5 =
\begin{pmatrix}
 0 & a_{3} & a_{5} & a_{6} & 0 \\
 -a_{3} & 0 & a_{8} & a_{9} & 0 \\
 -a_{5} & -a_{8} & 0 & a_{10} & 0 \\
 -a_{6} & -a_{9} & -a_{10} & 0 & 0 \\
 0 & 0 & 0 & 0 & 0 \\
\end{pmatrix}
\end{array}
\end{equation*}

{\small
\begin{equation*}
C[5]=B(A_1,\dots, A_5)=
{\def\arraycolsep{3.3pt}
\begin{pmatrix}
 0 & 0 & 0 & 0 & -a_{1} & -a_{2} & -a_{3} & -a_{4} & -a_{5} & -a_{6} \\
 0 & a_{1} & a_{2} & a_{3} & 0 & 0 & 0 & -a_{7} & -a_{8} & -a_{9} \\
 -a_{1} & 0 & a_{4} & a_{5} & 0 & a_{7} & a_{8} & 0 & 0 & -a_{10} \\
 -a_{2} & -a_{4} & 0 & a_{6} & -a_{7} & 0 & a_{9} & 0 & a_{10} & 0 \\
 -a_{3} & -a_{5} & -a_{6} & 0 & -a_{8} & -a_{9} & 0 & -a_{10} & 0 & 0 \\
\end{pmatrix}}
\end{equation*}}

\subsection{Rank and product table generation algorithms}

Once we are able to create $B(A_1,\dots, A_d)$ matrices for every dimension $d$, we have to study for which values of the parameters $a_i$ the rank is maximum, that is, the rank is $d$. In order to generate examples, this can be done by using any symbolic computational program (Mathematica among others). \medskip

One general way to compute the rank is choosing the minors of the matrix of size $d\times d$ and getting all the conditions the variables $a_i$ have to satisfy to have at least one minor whose determinant is not null, making the rank maximum. Therefore, the problem is simply finding the $\binom{\frac{d(d-1)}{2}}{d}$ minors and calculating their determinants. Unfortunately we are working with factorial complexity so increasing the dimension make the complexity extremely huge. For example, in the case $d=5$ we already have $252$ minors. Other way is using Gauss method, althought the posible nullity of variables might make it difficult too.\medskip

%
%

Following Theorem \ref{clasifcuadraticas}, the multiplication table for a quadratic reduced Lie algebra obtained from a $5$-quadratic family $\{A_1,\dots,A_5\}$ (so $\rank C[5]=5$) in the basis $\{v_1,\dots, v_{5}, z_1,\dots,z_5\}$ is
\begin{alignat*}{8} 
v_1 \wedge v_2 & =-a_{1} z_3 - a_{2} z_4 - a_{3} z_5,\qquad  &  v_2 \wedge v_4 & =-a_{2} z_1 + a_{7} z_3 - a_{9}  z_5, \\
v_1 \wedge v_3 & = a_{1} z_2 - a_{4} z_4 - a_{5} z_5,  &  v_2 \wedge v_5 & =-a_{3} z_1 + a_{8} z_3 + a_{9}  z_4, \\
v_1 \wedge v_4 & = a_{2} z_2 + a_{4} z_3 - a_{6} z_5,  &  v_3 \wedge v_4 & =-a_{4} z_1 - a_{7} z_2 - a_{10} z_5, \\
v_1 \wedge v_5 & = a_{3} z_2 + a_{5} z_3 + a_{6} z_4,  &  v_3 \wedge v_5 & =-a_{5} z_1 - a_{8} z_2 + a_{10} z_4, \\
v_2 \wedge v_3 & =-a_{1} z_1 - a_{7} z_4 - a_{8} z_5,  &  v_4 \wedge v_5 & =-a_{6} z_1 - a_{9} z_2 - a_{10} z_3. 
\end{alignat*}

In case $d=6$ the matrix $C[6] = B(A_1,\dots, A_6)$ is

{\footnotesize 
\begin{equation*}
\arraycolsep=1.11pt
    C[6] = 
\left(
\begin{array}{ccccccccccccccc}
 0 & 0 & 0 & 0 & 0 & -a_{1} & -a_{2} & -a_{3} & -a_{4} & -a_{5} & -a_{6} & -a_{7} & -a_{8} & -a_{9} & -a_{10} \\
 0 & a_{1} & a_{2} & a_{3} & a_{4} & 0 & 0 & 0 & 0 & -a_{11} & -a_{12} & -a_{13} & -a_{14} & -a_{15} & -a_{16} \\
 -a_{1} & 0 & a_{5} & a_{6} & a_{7} & 0 & a_{11} & a_{12} & a_{13} & 0 & 0 & 0 & -a_{17} & -a_{18} & -a_{19} \\
 -a_{2} & -a_{5} & 0 & a_{8} & a_{9} & -a_{11} & 0 & a_{14} & a_{15} & 0 & a_{17} & a_{18} & 0 & 0 & -a_{20} \\
 -a_{3} & -a_{6} & -a_{8} & 0 & a_{10} & -a_{12} & -a_{14} & 0 & a_{16} & -a_{17} & 0 & a_{19} & 0 & a_{20} & 0 \\
 -a_{4} & -a_{7} & -a_{9} & -a_{10} & 0 & -a_{13} & -a_{15} & -a_{16} & 0 & -a_{18} & -a_{19} & 0 & -a_{20} & 0 & 0 \\
\end{array}
\right),
\end{equation*}}
and gives use all the quadratic algebras described by the multiplication table (only those tables related to matrices $C[6]$ such that $\rank C[6]=6$ are valid).
{\fontsize{9}{11}\selectfont
\begin{alignat*}{8}
v_1 \wedge v_2 & =-a_{1} z_3 - a_{2} z_4-a_{3} z_5 - a_{4}  z_6,\qquad  &  v_2 \wedge v_6 & =-a_{4} z_1 + a_{13} z_3 + a_{15} z_4 + a_{16} z_5\\
v_1 \wedge v_3 & = a_{1} z_2 - a_{5} z_4-a_{6} z_5 - a_{7}  z_6,  & v_3 \wedge v_4&=-a_{5} z_1-a_{11} z_2-a_{17} z_5-a_{18} z_6\\
v_1 \wedge v_4 & = a_{2} z_2 + a_{5} z_3-a_{8} z_5 - a_{9}  z_6,  & v_3 \wedge v_5&=-a_{6} z_1-a_{1} z_2+a_{17} z_4-a_{19} z_6\\
v_1 \wedge v_5 & = a_{3} z_2 + a_{6} z_3+a_{8} z_4 - a_{10} z_6,  &v_3 \wedge v_6&=-a_{7} z_1-a_{13} z_2+a_{18} z_4+a_{19} z_5\\
v_1 \wedge v_6 & = a_{4} z_2 + a_{7} z_3+a_{9} z_4 + a_{10} z_5,  &  v_4 \wedge v_5&=-a_{8} z_1-a_{14} z_2-a_{17} z_3-a_{20} z_6\\
v_2 \wedge v_3 & =-a_{1} z_1 - a_{11} z_4 - a_{12} z_5-a_{13} z_6,  & v_4 \wedge v_6&=-a_{9} z_1-a_{15} z_2-a_{18} z_3+a_{20} z_5\\
v_2 \wedge v_4 & =-a_{2} z_1 + a_{11} z_3-a_{14} z_5 - a_{15} z_6, & v_5 \wedge v_6&=-a_{10} z_1-a_{16} z_2-a_{19} z_3-a_{20} z_4\\
 v_2 \wedge v_5&=-a_{3} z_1+a_{12} z_3+a_{14} z_4-a_{16} z_6& 
\end{alignat*}}
Note that the multiplication table can be easily obtained by means of the columns of the matrix $C[d]$.



\subsection{Final comments and examples}
In previous subsection 4.2 we have noted the limitations caused by the rank calculus of $C[d]$. In the following Table \ref{minimos} we can see the minimum number of variables needed to obtain maximum rank for small values of $d$.  The final Table \ref{C[d]} displays the matrices $C[d]$ for $d=5,6,7,8$.

\begin{table}[H]
\centering
\begin{tabular}{|l|l|}
\hline
Dimension & Condition to make the rank maximum\\\hline\hline
$d = 1$  & Invalid\\\hline
$d = 2$  & Invalid\\\hline
$d = 3$  & There is just one variable, $a_1$, and it has to be non null.\\\hline
$d = 4$  & Impossible.\\\hline
$d = 5,6$  & With just one non null variable we cannot obtain it.\\
&wee need at least two.\\\hline
$d = 7,8,9$  & With just two non null variables we cannot obtain it;\\
&wee need at least three.\\\hline
\end{tabular}

\caption{Minimun number of variables needed to get maximum rank in $C[d]$.}
\label{minimos}
\end{table}

\newcommand{\misep}{3.65}

\begin{sidewaystable}
\begin{table}[H]
\centering
\begin{tabular}{ll}
\hline
$d$ & $\quad C[d]$\\\hline
  &  \\
$5$  & {\tiny $\left(
{\def\arraycolsep{\misep pt}
\begin{array}{cccccccccc}
 0 & 0 & 0 & 0 & -a_{1} & -a_{2} & -a_{3} & -a_{4} & -a_{5} & -a_{6} \\
 0 & a_{1} & a_{2} & a_{3} & 0 & 0 & 0 & -a_{7} & -a_{8} & -a_{9} \\
 -a_{1} & 0 & a_{4} & a_{5} & 0 & a_{7} & a_{8} & 0 & 0 & -a_{10} \\
 -a_{2} & -a_{4} & 0 & a_{6} & -a_{7} & 0 & a_{9} & 0 & a_{10} & 0 \\
 -a_{3} & -a_{5} & -a_{6} & 0 & -a_{8} & -a_{9} & 0 & -a_{10} & 0 & 0 \\
\end{array}}
\right)$}\\
  &  \\
$6$  & {\tiny $\left(
{\def\arraycolsep{\misep pt}
\begin{array}{ccccccccccccccc}
 0 & 0 & 0 & 0 & 0 & -a_{1} & -a_{2} & -a_{3} & -a_{4} & -a_{5} & -a_{6} & -a_{7} & -a_{8} & -a_{9} & -a_{10} \\
 0 & a_{1} & a_{2} & a_{3} & a_{4} & 0 & 0 & 0 & 0 & -a_{11} & -a_{12} & -a_{13} & -a_{14} & -a_{15} & -a_{16} \\
 -a_{1} & 0 & a_{5} & a_{6} & a_{7} & 0 & a_{11} & a_{12} & a_{13} & 0 & 0 & 0 & -a_{17} & -a_{18} & -a_{19} \\
 -a_{2} & -a_{5} & 0 & a_{8} & a_{9} & -a_{11} & 0 & a_{14} & a_{15} & 0 & a_{17} & a_{18} & 0 & 0 & -a_{20} \\
 -a_{3} & -a_{6} & -a_{8} & 0 & a_{10} & -a_{12} & -a_{14} & 0 & a_{16} & -a_{17} & 0 & a_{19} & 0 & a_{20} & 0 \\
 -a_{4} & -a_{7} & -a_{9} & -a_{10} & 0 & -a_{13} & -a_{15} & -a_{16} & 0 & -a_{18} & -a_{19} & 0 & -a_{20} & 0 & 0 \\
\end{array}}
\right)$}\\
  &  \\
$7$  &{\tiny $\left(
{\def\arraycolsep{\misep pt}
\begin{array}{ccccccccccccccccccccc}
 0 & 0 & 0 & 0 & 0 & 0 & -a_{1} & -a_{2} & -a_{3} & -a_{4} & -a_{5} & -a_{6} & -a_{7} & -a_{8} & -a_{9} & -a_{10} & -a_{11} & -a_{12} & -a_{13} & -a_{14} & -a_{15} \\
 0 & a_{1} & a_{2} & a_{3} & a_{4} & a_{5} & 0 & 0 & 0 & 0 & 0 & -a_{16} & -a_{17} & -a_{18} & -a_{19} & -a_{20} & -a_{21} & -a_{22} & -a_{23} & -a_{24} & -a_{25} \\
 -a_{1} & 0 & a_{6} & a_{7} & a_{8} & a_{9} & 0 & a_{16} & a_{17} & a_{18} & a_{19} & 0 & 0 & 0 & 0 & -a_{26} & -a_{27} & -a_{28} & -a_{29} & -a_{30} & -a_{31} \\
 -a_{2} & -a_{6} & 0 & a_{10} & a_{11} & a_{12} & -a_{16} & 0 & a_{20} & a_{21} & a_{22} & 0 & a_{26} & a_{27} & a_{28} & 0 & 0 & 0 & -a_{32} & -a_{33} & -a_{34} \\
 -a_{3} & -a_{7} & -a_{10} & 0 & a_{13} & a_{14} & -a_{17} & -a_{20} & 0 & a_{23} & a_{24} & -a_{26} & 0 & a_{29} & a_{30} & 0 & a_{32} & a_{33} & 0 & 0 & -a_{35} \\
 -a_{4} & -a_{8} & -a_{11} & -a_{13} & 0 & a_{15} & -a_{18} & -a_{21} & -a_{23} & 0 & a_{25} & -a_{27} & -a_{29} & 0 & a_{31} & -a_{32} & 0 & a_{34} & 0 & a_{35} & 0 \\
 -a_{5} & -a_{9} & -a_{12} & -a_{14} & -a_{15} & 0 & -a_{19} & -a_{22} & -a_{24} & -a_{25} & 0 & -a_{28} & -a_{30} & -a_{31} & 0 & -a_{33} & -a_{34} & 0 & -a_{35} & 0 & 0 \\
\end{array}}
\right)$}\\
  &  \\
$8$  & {\tiny $\left(
{\def\arraycolsep{\misep pt}
\begin{array}{ccccccccccccccccccccc}
 0 & 0 & 0 & 0 & 0 & 0 & 0 & -a_{1} & -a_{2} & -a_{3} & -a_{4} & -a_{5} & -a_{6} & -a_{7} & -a_{8} & -a_{9} & -a_{10} & -a_{11} & -a_{12} & -a_{13} & -a_{14}\\
 0 & a_{1} & a_{2} & a_{3} & a_{4} & a_{5} & a_{6} & 0 & 0 & 0 & 0 & 0 & 0 & -a_{22} & -a_{23} & -a_{24} & -a_{25} & -a_{26} & -a_{27} & -a_{28} & -a_{29}\\
 -a_{1} & 0 & a_{7} & a_{8} & a_{9} & a_{10} & a_{11} & 0 & a_{22} & a_{23} & a_{24} & a_{25} & a_{26} & 0 & 0 & 0 & 0 & 0 & -a_{37} & -a_{38} & -a_{39}\\
 -a_{2} & -a_{7} & 0 & a_{12} & a_{13} & a_{14} & a_{15} & -a_{22} & 0 & a_{27} & a_{28} & a_{29} & a_{30} & 0 & a_{37} & a_{38} & a_{39} & a_{40} & 0 & 0 & 0\\
 -a_{3} & -a_{8} & -a_{12} & 0 & a_{16} & a_{17} & a_{18} & -a_{23} & -a_{27} & 0 & a_{31} & a_{32} & a_{33} & -a_{37} & 0 & a_{41} & a_{42} & a_{43} & 0 & a_{47} & a_{48}\\
 -a_{4} & -a_{9} & -a_{13} & -a_{16} & 0 & a_{19} & a_{20} & -a_{24} & -a_{28} & -a_{31} & 0 & a_{34} & a_{35} & -a_{38} & -a_{41} & 0 & a_{44} & a_{45} & -a_{47} & 0 & a_{50}\\
 -a_{5} & -a_{10} & -a_{14} & -a_{17} & -a_{19} & 0 & a_{21} & -a_{25} & -a_{29} & -a_{32} & -a_{34} & 0 & a_{36} & -a_{39} & -a_{42} & -a_{44} & 0 & a_{46} & -a_{48} & -a_{50} & 0 \\
 -a_{6} & -a_{11} & -a_{15} & -a_{18} & -a_{20} & -a_{21} & 0 & -a_{26} & -a_{30} & -a_{33} & -a_{35} & -a_{36} & 0 & -a_{40} & -a_{43} & -a_{45} & -a_{46} & 0 & -a_{49} & -a_{51} & -a_{52}\\
\end{array}}
\right.$}\\
  &  \\
& {\tiny $\left.
{\def\arraycolsep{\misep pt}
\begin{array}{ccccccc}
-a_{15} & -a_{16} & -a_{17} & -a_{18} & -a_{19} & -a_{20} & -a_{21} \\
-a_{30} & -a_{31} & -a_{32} & -a_{33} & -a_{34} & -a_{35} & -a_{36} \\
-a_{40} & -a_{41} & -a_{42} & -a_{43} & -a_{44} & -a_{45} & -a_{46} \\
0       & -a_{47} & -a_{48} & -a_{49} & -a_{50} & -a_{51} & -a_{52} \\
 a_{49} & 0       & 0       & 0       & -a_{53} & -a_{54} & -a_{55} \\
 a_{51} & 0       &  a_{53} &  a_{54} & 0       & 0       & -a_{56} \\
 a_{52} & -a_{53} & 0       &  a_{55} & 0       &  a_{56} & 0       \\
 0      & -a_{54} & -a_{55} & 0       & -a_{56} & 0       & 0       \\
\end{array}}
\right)$}\\
  &  \\
\hline
\end{tabular}
\caption{Matrices $B(A_1,\dots,A_d)$ attached to a $d$-quadratic family of matrices $\{A_1,\dots,A_d\}$.}
\label{C[d]}
\end{table}
\end{sidewaystable}

For example if $d=6$, among its 20 variables, only the couples $(a_i, a_{21-i})$ (both variables being non null) with the rest of them null, make the rank maximum. In the case $d=5$ each variable appears in three possible couples that makes the rank maximum, while in $n=7$ there is no couples of variables, we need at least three non null parameters to get rank 7. For $d=7,8,9$ the non null triples $(a_1,a_{10},a_{15})$, $(a_1, a_{12},a_{56})$ and $(a_1,a_{65},a_{85})$ provide quadratic Lie algebras of types $7$, $8$ and $9$ respectively.\medskip

According to Corollary 3.6 in \cite{NoRe97}, there is a finite number of quadratic $2$-step nilpotent Lie algebras of dimension $\leq 17$. The complete list of these algebras is unknown (in \cite{Ka07} we can find those up to dimension $10$) and it is encoded in the matrices listed in Table \ref{C[d]}. The real Lie groups associated to these algebras provide examples of compact pseudo-Riemannian nilmanifolds.

\bigskip

\noindent{\bf Acknowledgments.} Pilar Benito and Daniel de-la-Concepci\'on acknowledge financial support from the Spanish Ministerio de Econom\'ia y Com\-pe\-ti\-ti\-vi\-dad (MTM 2013-45588-CO3-3). Daniel de-la-Concepci\'on also thanks support from Spanish FPU Grant 12/03224.
\bigskip


\vfill
{\noindent
\author{Pilar Benito, Daniel de-la-Concepci\'on, Jorge Rold\'an-L\'opez, Iciar Sesma 
\\{\small Departamento de Matem\'aticas y Computaci\'on}\\
{\small  Universidad de La Rioja}\\
{\small  26004, Logro\~no. Spain}\\
{\small pilar.benito@unirioja.es}, {\small daniel-de-la.concepcion@unirioja.es},  \\{\small jorge.roldanl@unirioja.es}, {\small icsesma@unirioja.es}

\end{document}